\documentclass[reqno,11pt]{article} 
\usepackage{geometry}
\usepackage{color}
\usepackage{tikz}
\usepackage{amssymb, amsmath, amsthm}
\usepackage{graphicx}
\usepackage{float}
\newcommand{\Mod}[1]{\ (\textup{mod}\ #1)}
\theoremstyle{plain} 
\newtheorem{theorem}{\indent\sc Theorem}[section]
\newtheorem{lemma}[theorem]{\indent\sc Lemma}

\newtheorem{proposition}[theorem]{\indent\sc Proposition}

\theoremstyle{definition} 

\makeatletter
\def\address#1#2{\begingroup
\noindent\parbox[t]{7.8cm}{%
\small{\scshape\ignorespaces#1}\par\vskip1ex
\noindent\small{\itshape E-mail address}%
\/: #2\par\vskip4ex}\hfill%
\endgroup}%
\makeatother
\title{Normal bases for modular function fields} 
\author{
\textsc{Ja Kyung Koo, Dong Hwa Shin and Dong Sung Yoon$^*$} 
}
\date{} 
\begin{document}

\maketitle

\allowdisplaybreaks

\footnote{ 
2010 \textit{Mathematics Subject Classification}. Primary 11F03, Secondary 11G16.}
\footnote{ 
\textit{Key words and phrases}. modular functions, modular units, normal bases}
\footnote{$^*$Corresponding author.\\
\thanks{
The second named author was supported by Hankuk University of Foreign Studies Research Fund of 2016.} }

\begin{abstract}
We provide a concrete example of a normal basis for a finite Galois extension which is not abelian. More precisely, let
$\mathbb{C}(X(N))$ be the field of meromorphic functions on the modular curve
$X(N)$ of level $N$. We construct a completely free element in the extension
$\mathbb{C}(X(N))/\mathbb{C}(X(1))$ by means of
Siegel functions.
\end{abstract}

\section {Introduction}

Let $E$ be a finite Galois extension of a field $F$ with
\begin{equation*}
G=\mathrm{Gal}(E/F)=\{\sigma_1,\,\sigma_2,\,\ldots,\,\sigma_n\}.
\end{equation*}
The well-known normal basis theorem (\cite{vanderWaerden}) states that
there always exists an element $a$ of $E$ for which
\begin{equation*}
\{a^{\sigma_1},\,a^{\sigma_2},\,\ldots,\,a^{\sigma_n}\}
\end{equation*}
becomes a basis for $E$ over $F$. 
We call such a basis a \textit{normal basis} for the extension $E/F$, and say that the element $a$ is \textit{free} in $E/F$. 
In other words, $E$ is a free $F[G]$-module of rank $1$ generated by $a$.
Moreover, Blessenohl and Johnson proved in \cite{B-J} that
there is a primitive element $a$ for $E/F$ which
is free in $E/L$ for every intermediate field $L$ of $E/F$.
Such an element $a$ is said to be \textit{completely free} in the extension $E/F$.
However, not much is known so far about explicit construction of
(completely) free elements when $F$ is infinite. 
As for number fields, we refer to
\cite{Hachenberger}, \cite{J-K-S}, \cite{K-S}, \cite{Leopoldt}, \cite{Okada}, \cite{Schertz}, \cite{Taylor}.
\par
For a positive integer $N$, let
\begin{equation*}
\Gamma(N)=\{\sigma\in\mathrm{SL}_2(\mathbb{Z})~|~\sigma\equiv
I_2\Mod{N\cdot M_2(\mathbb{Z})}\}
\end{equation*}
be the principal congruence subgroup of $\mathrm{SL}_2(\mathbb{Z})$ of level $N$
which acts on the upper half-plane $\mathbb{H}=\{\tau\in\mathbb{C}~|~
\mathrm{Im}(\tau)>0\}$ by fractional linear transformations. Corresponding to $\Gamma(N)$, let
\begin{equation*}
X(N)=\Gamma(N)\backslash\mathbb{H}^*
\end{equation*}
be the modular curve of level $N$, where
$\mathbb{H}^*=\mathbb{H}\cup\mathbb{Q}\cup\{\mathrm{i}\infty\}$
(\cite[Chapter 1]{Shimura}).
We denote its meromorphic function fields by $\mathbb{C}(X(N))$.
As is well known, $\mathbb{C}(X(N))$
is a Galois extension of $\mathbb{C}(X(1))$ with
\begin{equation}\label{Galois}
\mathrm{Gal}\left(\mathbb{C}(X(N))/\mathbb{C}(X(1))\right)\simeq
\Gamma(1)/\pm\Gamma(N)\simeq\mathrm{SL}_2(\mathbb{Z}/N\mathbb{Z})/\{\pm I_2\}
\end{equation}
(\cite[Theorem 2 in Chapter 6]{Lang} and \cite[Proposition 6.1]{Shimura}).
Furthermore, if $N\geq2$, then $\mathbb{C}(X(N))$ is not
an abelian extension of $\mathbb{C}(X(1))$. In this paper,
we shall find a completely free element $g(\tau)$ in
$\mathbb{C}(X(N))/\mathbb{C}(X(1))$
in terms of Siegel functions (Theorem \ref{main}).
This gives a concrete example of a normal basis
for a nonabelian Galois extension.

\section {Siegel functions as modular functions}

We shall briefly introduce Siegel functions and their basic properties, and further
develop a couple of lemmas for later use.
\par
For a lattice $\Lambda$ in $\mathbb{C}$,
the \textit{Weierstrass $\sigma$-function} relative to $\Lambda$
is defined by
\begin{equation*}
\sigma(z;\,\Lambda)=z\prod_{\lambda\in \Lambda\setminus\{0\}}\left(1-\frac{z}{\lambda}\right)
e^{z/\lambda+(1/2)(z/\lambda)^2}\quad(z\in\mathbb{C}).
\end{equation*}
Taking logarithmic derivative, we obtain the \textit{Weierstrass
$\zeta$-function}
\begin{equation*}
\zeta(z;\,\Lambda)=\frac{\sigma'(z;\,\Lambda)}{\sigma(z;\,\Lambda)}
=\frac{1}{z}+\sum_{\lambda\in
\Lambda\setminus\{0\}}\left(\frac{1}{z-\lambda}+\frac{1}{\lambda}+
\frac{z}{\lambda^2}\right)\quad(z\in\mathbb{C}).
\end{equation*}
One can readily see that
\begin{equation*}
\zeta'(z;\,\Lambda)=-\frac{1}{z^2}+\sum_{\lambda\in\Lambda\setminus\{0\}}\left(
-\frac{1}{(z-\lambda)^2}+\frac{1}{\lambda^2}\right),
\end{equation*}
which is periodic with respect to $\Lambda$. Thus,
for each $\lambda\in\Lambda$, there is a constant $\eta(\lambda;\,\Lambda)$ such that
\begin{equation*}
\zeta(z+\lambda;\,\Lambda)-\zeta(z;\,\Lambda)=\eta(\lambda;\,\Lambda)\quad(z\in\mathbb{C}).
\end{equation*}
Now, for $\mathbf{v}=\begin{bmatrix}v_1\\v_2
\end{bmatrix}\in\mathbb{Q}^2\setminus\mathbb{Z}^2$, we define
the \textit{Siegel function} $g_\mathbf{v}(\tau)$ by
\begin{equation*}
g_\mathbf{v}(\tau)=\exp\left(-(1/2)(v_1\eta(\tau;\,[\tau,\,1])+
v_2\eta(1;\,[\tau,\,1]))
(v_1\tau+v_2)\right)
\sigma(v_1\tau+v_2;\,[\tau,\,1])\eta(\tau)^2
\quad(\tau\in\mathbb{H}),
\end{equation*}
where $[\tau,\,1]=\tau\mathbb{Z}+\mathbb{Z}$ and
$\eta(\tau)$ is the \textit{Dedekind $\eta$-function} given by
\begin{equation}\label{eta}
\eta(\tau)=\sqrt{2\pi}e^{\pi\mathrm{i}/4}q^{1/24}\prod_{n=1}^\infty
(1-q^n)\quad(q=e^{2\pi\mathrm{i}\tau},\,\tau\in\mathbb{H}).
\end{equation}
\par
Let
\begin{equation*}
\mathbf{B}_2(x)=x^2-x+\frac{1}{6}\quad (x\in\mathbb{R})
\end{equation*}
be the second Bernoulli  polynomial, and let $\langle x\rangle$ be
the fractional part of $x$ in the interval $[0,\,1)$.

\begin{proposition}\label{properties}
Let $\mathbf{v}=\begin{bmatrix}v_1\\v_2\end{bmatrix}\in(1/N)\mathbb{Z}^2\setminus\mathbb{Z}^2$
for an integer $N\geq2$. 
\begin{enumerate}
\item[\textup{(i)}] $g_\mathbf{v}(\tau)$ has the infinite product expansion 
\begin{equation*}
g_\mathbf{v}(\tau)
=-e^{\pi\mathrm{i}v_2(v_1-1)}q^{(1/2)\mathbf{B}_2(v_1)}
(1-q^{v_1}e^{2\pi\mathrm{i}v_2})\prod_{n=1}^\infty (1-q^{n+v_1}e^{2\pi\mathrm{i}v_2})
(1-q^{n-v_1}e^{-2\pi\mathrm{i}v_2})
\end{equation*}
with respect to $q=e^{2\pi\mathrm{i}\tau}$.
\item[\textup{(ii)}] We have the $q$-order formula
\begin{equation*}
\mathrm{ord}_q~g_\mathbf{v}(\tau)=\frac{1}{2}\mathbf{B}_2(\langle v_1\rangle).
\end{equation*}
\item[\textup{(iii)}] $g_\mathbf{v}(\tau)^{12N}$ belongs to $\mathbb{C}(X(N))$ and
has neither zeros nor poles on $\mathbb{H}$.
\item[\textup{(iv)}] $g_\mathbf{v}(\tau)^{12N}$ depends only on $\pm\mathbf{v}\Mod{\mathbb{Z}^2}$,
    and satisfies
\begin{equation*}
\left(g_\mathbf{v}(\tau)^{12N}\right)^\sigma=
(g_\mathbf{v}^{12N}\circ\sigma)(\tau)=
g_{\sigma^T\mathbf{v}}(\tau)^{12N}\quad(\sigma\in\mathrm{SL}_2(\mathbb{Z})),
\end{equation*}
where $\sigma^T$ stands for the transpose of $\sigma$.
\end{enumerate}
\end{proposition}
\begin{proof}
\begin{enumerate}
\item[(i)] See \cite[K 4 in p. 29]{K-L}.
\item[(ii)] See \cite[p. 31]{K-L}.
\item[(iii)] See \cite[Theorem 1.2 in Chapter 2]{K-L}.
\item[(iv)] See \cite[Proposition 1.3 in Chapter 2]{K-L}.
\end{enumerate}
\end{proof}

For a positive integer $N$, let $\Gamma_1(N)$ be the congruence subgroup of $\mathrm{SL}_2(\mathbb{Z})$ defined by
\begin{equation*}
\Gamma_1(N)=\left\{\sigma\in\mathrm{SL}_2(\mathbb{Z})~|~
\sigma\equiv\begin{bmatrix}1&\mathrm{*}\\0&1\end{bmatrix}\Mod{N\cdot M_2(\mathbb{Z})}\right\}.
\end{equation*}
Now, we let $N\geq2$, and consider the function 
\begin{equation*}
g(\tau)=g_{\left[\begin{smallmatrix}
0\\1/N\end{smallmatrix}\right]}(\tau)^{-12N\ell}
g_{\left[\begin{smallmatrix}
1/N\\0\end{smallmatrix}\right]}(\tau)^{-12Nm}
\end{equation*}
where $\ell$ and $m$ are integers such that $\ell>m>0$.
Then we see from Proposition \ref{properties} (iii) that $g(\tau)$ belongs to $\mathbb{C}(X(N))$.

\begin{lemma}\label{minimumorder}
We have
\begin{equation*}
\mathrm{ord}_q\left(\frac{g(\tau)^\sigma}{g(\tau)}\right)\geq0
\quad\textrm{for all}~\sigma\in\mathrm{SL}_2(\mathbb{Z}).
\end{equation*}
The equality holds if and only if $\sigma\in\pm\Gamma_1(N)$.
\end{lemma}
\begin{proof}
Let $\sigma=\begin{bmatrix}a&b\\c&d\end{bmatrix}\in\mathrm{SL}_2(\mathbb{Z})$.
Note that $a\equiv c\equiv0\Mod{N}$ is impossible. 
We get by Proposition \ref{properties} (iv) and (ii) that
\begin{eqnarray*}
\mathrm{ord}_q\left(\frac{g(\tau)^\sigma}{g(\tau)}\right)&=&
\mathrm{ord}_q\left(\frac{g_{\left[\begin{smallmatrix}
c/N\\d/N\end{smallmatrix}\right]}(\tau)^{-12N\ell}
g_{\left[\begin{smallmatrix}
a/N\\b/N\end{smallmatrix}\right]}(\tau)^{-12Nm}}
{g_{\left[\begin{smallmatrix}
0\\1/N\end{smallmatrix}\right]}(\tau)^{-12N\ell}
g_{\left[\begin{smallmatrix}
1/N\\0\end{smallmatrix}\right]}(\tau)^{-12Nm}}\right)\\
&=&6N\left(\ell\mathbf{B}_2(0)+m\mathbf{B}_2(1/N)-\ell\mathbf{B}_2(\langle c/N\rangle)
-m\mathbf{B}_2(\langle a/N\rangle)\right).
\end{eqnarray*}
Then we deduce
from the fact $\ell>m>0$ and Figure \ref{figure1} that
\begin{equation*}
\mathrm{ord}_q\left(\frac{g(\tau)^\sigma}{g(\tau)}\right)\geq0
\end{equation*}
with equality if and only if
\begin{equation}\label{equality}
\langle c/N\rangle=0\quad\textrm{and}\quad\langle a/N\rangle=1/N~\textrm{or}~1-1/N.
\end{equation}
Moreover, by the relation $\det(\sigma)=ad-bc=1$ one can express the condition (\ref{equality}) as
\begin{equation*}
\sigma\equiv\pm\begin{bmatrix}1&\mathrm{*}\\0&1\end{bmatrix}\Mod{N\cdot M_2(\mathbb{Z})}.
\end{equation*}
This proves the lemma.
\end{proof}

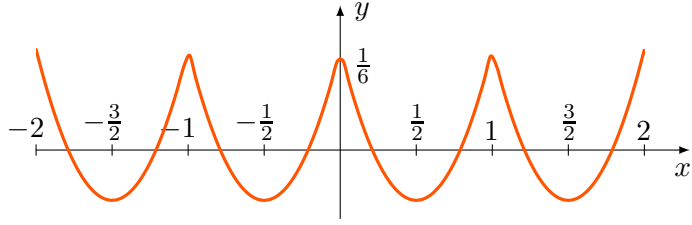
\begin{figure}[H]
\centering
\definecolor{ffvvqq}{rgb}{1.,0.3333333333333333,0.}
\begin{tikzpicture}[line cap=round,line join=round,>=latex,x=2.0cm,y=8.0cm]
\draw[->,color=black] (-2.,0.) -- (2.3,0.);
\foreach \x in {-2.,-1.5,-1.,-0.5,0.5,1.,1.5,2.}
\draw[shift={(\x,0)},color=black] (0pt,2pt) -- (0pt,-2pt);
\draw[color=black] (2.13,-0.06) node [anchor=south west] {$x$};
\draw[color=black] (-0.76,0) node [anchor=south west] {$-\frac{1}{2}$};
\draw[color=black] (-1.26,0) node [anchor=south west] {$-1$};
\draw[color=black] (-1.76,0) node [anchor=south west] {$-\frac{3}{2}$};
\draw[color=black] (-2.26,0) node [anchor=south west] {$-2$};
\draw[color=black] (0.38,0) node [anchor=south west] {$\frac{1}{2}$};
\draw[color=black] (0.88,0) node [anchor=south west] {$1$};
\draw[color=black] (1.38,0) node [anchor=south west] {$\frac{3}{2}$};
\draw[color=black] (1.88,0) node [anchor=south west] {$2$};
\draw[->,color=black] (0.,-0.11358885017421595) -- (0.,0.24);
\foreach \y in {}
\draw[shift={(0,\y)},color=black] (2pt,0pt) -- (-2pt,0pt);
\draw[color=black] (0.025806451612903226,0.22874564459930322) node [anchor=west] {$y$};
\draw[color=black] (0.025806451612903226,0.14) node [anchor=west] {$\frac{1}{6}$};
\clip(-2.,-0.11358885017421595) rectangle (2.,0.24);
\draw[line width=1.2pt,color=ffvvqq,smooth,samples=100,domain=-2:2] plot(\x,{((\x)-floor((\x)))^(2.0)-((\x)-floor((\x)))+1.0/6.0});
\end{tikzpicture}
\caption{The graph of $y=\mathbf{B}_2(\langle x\rangle)$}\label{figure1}
\end{figure}

Let $\mathbb{R}_+$ denote the set of positive real numbers.

\begin{lemma}\label{evaluation}
Given any $\varepsilon\in\mathbb{R}_+$, we can take $r\in\mathbb{R}_+$
and an integer $m$ large enough so that
\begin{equation*}
\left|\frac{g^\sigma(r\mathrm{i})}{g(r\mathrm{i})}
\right|<\varepsilon\quad\textrm{for all}~\sigma
\in\mathrm{SL}_2(\mathbb{Z})\setminus\pm\Gamma(N).
\end{equation*}
\end{lemma}
\begin{proof}
First, consider the case where $\sigma\not\in\pm\Gamma_1(N)$.
Then, we obtain by Lemma \ref{minimumorder} that
\begin{equation*}
\mathrm{ord}_q\left(\frac{g(\tau)^\sigma}{g(\tau)}\right)>0,
\end{equation*}
which implies that $g(\tau)^\sigma/g(\tau)$ has a zero
at the cusp $\mathrm{i\infty}$.
Hence we can take $r_\sigma\in\mathbb{R}_+$ sufficiently large so as to have
\begin{equation*}
\left|\frac{g^\sigma(r_\sigma\mathrm{i})}{g(r_\sigma\mathrm{i})}\right|<\varepsilon.
\end{equation*}
Set
\begin{equation*}
r=\max\left\{r_\sigma~|~\sigma\in\mathrm{SL}_2(\mathbb{Z})\setminus\pm\Gamma_1(N)\right\}.
\end{equation*}
\par
Second, let $\sigma\in\pm\Gamma_1(N)\setminus\pm\Gamma(N)$, and so $\sigma\equiv\pm
\begin{bmatrix}1&b\\0&1\end{bmatrix}\Mod{N\cdot M_2(\mathbb{Z})}$ for some $b\in\mathbb{Z}$
with $b\not\equiv0\Mod{N}$. We then derive that
\begin{eqnarray*}
\left|\frac{g^\sigma(r\mathrm{i})}{g(r\mathrm{i})}\right|&=&
\left|\frac{g_{\left[\begin{smallmatrix}
0\\1/N\end{smallmatrix}\right]}(r\mathrm{i})^{-12N\ell}
g_{\left[\begin{smallmatrix}
1/N\\b/N\end{smallmatrix}\right]}(r\mathrm{i})^{-12Nm}}
{g_{\left[\begin{smallmatrix}
0\\1/N\end{smallmatrix}\right]}(r\mathrm{i})^{-12N\ell}
g_{\left[\begin{smallmatrix}
1/N\\0\end{smallmatrix}\right]}(r\mathrm{i})^{-12Nm}}
\right|\quad\textrm{by Proposition \ref{properties} (iv)}\\
&=&\left|\frac{g_{\left[\begin{smallmatrix}
1/N\\0\end{smallmatrix}\right]}(r\mathrm{i})}
{g_{\left[\begin{smallmatrix}
1/N\\b/N\end{smallmatrix}\right]}(r\mathrm{i})}
\right|^{12Nm}\\
&=&\left|\frac{1-R^{1/N}}{1-R^{1/N}\zeta_N^b}\right|^{12Nm}
    \prod_{n=1}^\infty\left|\frac{(1-R^{n+1/N})(1-R^{n-1/N})}
    {(1-R^{n+1/N}\zeta_N^b)(1-R^{n-u}\zeta_N^{-b})}\right|^{12Nm}\\
    &&\textrm{by Proposition \ref{properties} (i), where $R=e^{-2\pi r}$
    and $\zeta_N=e^{2\pi\mathrm{i}/N}$}\\
&\leq&\left|\frac{1-R^{1/N}}{1-R^{1/N}\zeta_N^b}\right|^{12Nm}\\
&&\textrm{because $\left|1-x\right|\leq\left|1-x\zeta\right|$
for any $x\in\mathbb{R}_+$ with $x<1$ and any root of unity $\zeta$}.
\end{eqnarray*}
Therefore, if $m$ is sufficiently large, then we attain
\begin{equation*}
\left|\frac{g^\sigma(r\mathrm{i})}{g(r\mathrm{i})}\right|~<~\varepsilon.
\end{equation*}
This completes the proof.
\end{proof}

\section {Completely free elements
in modular function fields}

Let $N\geq2$. In this section, we shall show that
\begin{equation*}
g(\tau)=g_{\left[\begin{smallmatrix}
0\\1/N\end{smallmatrix}\right]}(\tau)^{-12N\ell}
g_{\left[\begin{smallmatrix}
1/N\\0\end{smallmatrix}\right]}(\tau)^{-12Nm}\quad\textrm{with}~
\ell>m>0
\end{equation*}plays an important role as completely normal elements in
modular function field extensions.

\begin{proposition}\label{primitive}
The function $g(\tau)$ generates $\mathbb{C}(X(N))$ over $\mathbb{C}(X(1))$.
\end{proposition}
\begin{proof}
Suppose that $\sigma=\begin{bmatrix}a&b\\c&d\end{bmatrix}\in\mathrm{SL}_2(\mathbb{Z})$ leaves
$g(\tau)$ fixed. In particular, since
$\mathrm{ord}_q~g(\tau)=\mathrm{ord}_q~g(\tau)^\sigma$,
we get by Lemma \ref{minimumorder} that
$\sigma\in\pm\Gamma_1(N)$.
Furthermore, we see by Proposition \ref{properties} (iv) and (ii) that
\begin{eqnarray*}
\mathrm{ord}_q~g(\tau)^{\left[\begin{smallmatrix}0&-1\\1&0\end{smallmatrix}\right]}
&=&\mathrm{ord}_q\left(g_{\left[\begin{smallmatrix}
1/N\\0\end{smallmatrix}\right]}(\tau)^{-12N\ell}
g_{\left[\begin{smallmatrix}
0\\-1/N\end{smallmatrix}\right]}(\tau)^{-12Nm}\right)\\
&=&-6N\ell\mathbf{B}_2(1/N)-6Nm\mathbf{B}_2(0)\\
&=&\mathrm{ord}_q~\left(g(\tau)^\sigma\right)^{\left[\begin{smallmatrix}0&-1\\1&0\end{smallmatrix}\right]}\\
&=&\mathrm{ord}_q~g(\tau)^{\left[\begin{smallmatrix}b&-a\\d&-c\end{smallmatrix}\right]}\\
&=&\mathrm{ord}_q\left(g_{\left[\begin{smallmatrix}
d/N\\-c/N\end{smallmatrix}\right]}(\tau)^{-12N\ell}
g_{\left[\begin{smallmatrix}
b/N\\-a/N\end{smallmatrix}\right]}(\tau)^{-12Nm}\right)\\
&=&-6N\ell\mathbf{B}_2(\langle d/N\rangle)-6Nm\mathbf{B}_2(\langle b/N\rangle).
\end{eqnarray*}
Thus we obtain $b\equiv0\Mod{N}$, and hence $\sigma\in\pm\Gamma(N)$. Therefore,
we conclude by (\ref{Galois}) and the Galois theory that $g(\tau)$ generates $\mathbb{C}(X(N))$ over $\mathbf{C}(X(1))$.
\end{proof}

\begin{theorem}\label{submain}
Let $X^0(N)$ be the modular
curve for the congruence subgroup
\begin{equation*}
\Gamma^0(N)=\left\{\sigma\in\mathrm{SL}_2(\mathbb{Z})~|~
\sigma\equiv\begin{bmatrix}\mathrm{*}&0\\\mathrm{*}&\mathrm{*}\end{bmatrix}\Mod{N\cdot M_2(\mathbb{Z})}\right\}
\end{equation*}
with meromorphic function field $\mathbb{C}(X^0(N))$.
Then, $g(\tau)$
is completely free in $\mathbb{C}(X(N))/
\mathbb{C}(X^0(N))$.
\end{theorem}
\begin{proof}
Note that $\mathbb{C}(X(N))$ is a Galois extension of
$\mathbb{C}(X^0(N))$ with
\begin{equation*}
\mathrm{Gal}\left(\mathbb{C}(X(N))/\mathbb{C}(X^0(N))\right)
\simeq\Gamma^0(N)/\pm\Gamma(N).
\end{equation*}
It follows from Proposition \ref{primitive} that $g(\tau)$ generates
$\mathbb{C}(X(N))$ over $\mathbb{C}(X^0(N))$.
\par
Now, let $L$ be any intermediate field of $\mathbb{C}(X(N))/
\mathbb{C}(X^0(N))$ with
\begin{equation*}
\mathrm{Gal}\left(\mathbb{C}(X(N))/L)\right)=\{\sigma_1=\mathrm{Id},\,\sigma_2,\,\ldots,\,\sigma_k\}.
\end{equation*}
Since
$\Gamma^0(N)\cap\pm\Gamma_1(N)=\pm\Gamma(N)$,
we must have
\begin{equation}\label{notin}
\sigma_i\not\in\pm\Gamma_1(N)\quad(i=2,\,\ldots,\,k).
\end{equation}
Set
\begin{equation*}
g_i=g(\tau)^{\sigma_i}\quad(i=1,\,2,\ldots,\,k),
\end{equation*}
and suppose that
\begin{equation}\label{linear}
c_1g_1+c_2g_2+\cdots+c_kg_k=0\quad
\textrm{for some}~c_1,\,c_2,\,\ldots,\,c_k\in L.
\end{equation}
Acting each $\sigma_i$ ($i=1,\,2,\,\ldots,\,k$) on both sides of (\ref{linear}) we achieve the system of equations
\begin{equation*}
\left\{\begin{array}{ccc}
c_1g_1^{\sigma_1}+c_2g_2^{\sigma_1}+\cdots+c_kg_k^{\sigma_1}&=&0,\\
c_1g_1^{\sigma_2}+c_2g_2^{\sigma_2}+\cdots+c_kg_k^{\sigma_2}&=&0,\\
&\vdots&\\\
c_1g_1^{\sigma_k}+c_2g_2^{\sigma_k}+\cdots+c_kg_k^{\sigma_k}&=&0,
\end{array}\right.
\end{equation*}
which can be rewritten as
\begin{equation*}
A\begin{bmatrix}c_1\\c_2\\\vdots\\c_k\end{bmatrix}=
\begin{bmatrix}0\\0\\\vdots\\0\end{bmatrix}\quad\textrm{with}~
A=\left[g_j^{\sigma_i}\right]_{1\leq i,\,j\leq k}.
\end{equation*}
Letting $S_k$ be the permutation group on $\{1,\,2,\,\cdots,\,k\}$, we derive that
\begin{eqnarray*}
\det(A)&=&\sum_{j_1j_2\cdots j_k\in S_k}\mathrm{sgn}(j_1j_2\cdots j_k)
g_{j_1}^{\sigma_1}g_{j_2}^{\sigma_2}\cdots g_{j_k}^{\sigma_k}\\
&=&\pm g^k+\sum_{\begin{smallmatrix}j_1j_2\cdots j_k\in S_k~\textrm{such that}\\
\sigma_{j_1}\sigma_{j_2}\cdots\sigma_{j_k}\neq \sigma_1^{-1}\sigma_2^{-1}\cdots\sigma_k^{-1}
\end{smallmatrix}}\pm g^{\sigma_{j_1}\sigma_1}g^{\sigma_{j_2}\sigma_2}
\cdots g^{\sigma_{j_k}\sigma_k}\\
&=&\pm g^k\left(1+\sum_{\begin{smallmatrix}j_1j_2\cdots j_k\in S_k~\textrm{such that}\\
\sigma_{j_1}\sigma_{j_2}\cdots\sigma_{j_k}\neq \sigma_1^{-1}\sigma_2^{-1}\cdots\sigma_k^{-1}
\end{smallmatrix}}\pm\left(\frac{g^{\sigma_{j_1}\sigma_1}}{g}\right)\left(\frac{g^{\sigma_{j_2}\sigma_2}}{g}\right)
\cdots\left(\frac{g^{\sigma_{j_k}\sigma_k}}{g}\right)\right).
\end{eqnarray*}
Observe that for each 
$j_1j_2\cdots j_k\in S_k$ with
$\sigma_{j_1}\sigma_{j_2}\cdots\sigma_{j_k}
\neq\sigma_1^{-1}\sigma_2^{-1}\cdots\sigma_k^{-1}$,
we have 
\begin{equation*}
\sigma_{j_i}\sigma_i\neq\mathrm{Id}\quad\textrm{for some}~1\leq i\leq k.
\end{equation*}
Thus we attain that
\begin{eqnarray*}
\mathrm{ord}_q~\det(A)&=&\mathrm{ord}_q~g^k\quad\textrm{by
(\ref{notin}) and Lemma \ref{minimumorder}}\\
&=&-6kN\left(\ell\mathbf{B}_2(0)+m\mathbf{B}_2(1/N)\right)\quad\textrm{by Proposition \ref{properties} (ii)}\\
&<&0\quad\textrm{by the fact $\ell>m>0$ and Figure \ref{figure1}},
\end{eqnarray*}
which implies that
\begin{equation*}
\det(A)\neq0\quad\textrm{and}\quad c_1=c_2=\cdots=c_k=0.
\end{equation*}
Therefore, $\{g_1,\,g_2,\,\ldots,\,g_k\}$ is linearly independent over $L$;
and hence $g(\tau)$ is completely free in $\mathbb{C}(X(N))/\mathbb{C}(X^0(N))$.
\end{proof}

\begin{theorem}\label{main}
There is a positive integer $M$ for which 
\begin{equation*}
g(\tau)=g_{\left[\begin{smallmatrix}
0\\1/N\end{smallmatrix}\right]}(\tau)^{-12N\ell}
g_{\left[\begin{smallmatrix}
1/N\\0\end{smallmatrix}\right]}(\tau)^{-12Nm}
\end{equation*}
is completely free in $\mathbb{C}(X(N))/\mathbb{C}(X(1))$ for $\ell>m>M$.
\end{theorem}
\begin{proof}
Let $d=[\mathbb{C}(X(N)):\mathbb{C}(X(1))]$. 
We see from Lemma \ref{evaluation} and (\ref{Galois}) that there exist a positive integer $M$ and $r\in\mathbb{R}_+$ so that
if $\ell>m>M$, then 
\begin{equation}\label{lmm}
\left|\frac{g^{\sigma}(r\mathrm{i})}{g(r\mathrm{i})}\right|~<~\frac{1}{d!-1}\quad
\textrm{for all $\sigma\in\mathrm{Gal}\left(\mathbb{C}(X(N))/\mathbb{C}(X(1))\right)$ with
$\sigma\neq\mathrm{Id}$}. 
\end{equation}
Now, let $\ell>m>M$.
Let $L$ be any intermediate field of $\mathbb{C}(X(N))/\mathbb{C}(X(1))$ with
\begin{equation*}
\mathrm{Gal}(\mathbb{C}(X(N))/L)=\{\sigma_1=\mathrm{Id},\,\sigma_2,\,\ldots,\,\sigma_n\}.
\end{equation*}
Then it follows from Proposition \ref{primitive} that $g(\tau)$ generates
$\mathbb{C}(X(N))$ over $L$.
Consider the $n\times n$ matrix
\begin{equation*}
B=\left[g_j^{\sigma_i}\right]_{1\leq i,\,j\leq n}\quad
\textrm{where}~g_j=g(\tau)^{\sigma_j}.
\end{equation*}
As in Theorem \ref{submain} it suffices to show  $\det(B)\neq0$ in order to 
prove that $\{g_1,\,g_2,\,\ldots,\,g_n\}$ is linearly independent over $L$. 
We derive that
\begin{eqnarray*}
\left|\det(B)(r\mathrm{i})\right|&=&\left|\sum_{j_1j_2\cdots j_n\in S_n}\mathrm{sgn}(j_1j_2\cdots j_n)
g_{j_1}^{\sigma_1}(r\mathrm{i})g_{j_2}^{\sigma_2}(r\mathrm{i})\cdots g_{j_n}^{\sigma_n}(r\mathrm{i})\right|\\
&=&\left|\pm g(r\mathrm{i})^n+\sum_{\begin{smallmatrix}j_1j_2\cdots j_n\in S_n~\textrm{such that}\\
\sigma_{j_1}\sigma_{j_2}\cdots\sigma_{j_n}\neq \sigma_1^{-1}\sigma_2^{-1}\cdots\sigma_n^{-1}
\end{smallmatrix}}\pm g^{\sigma_{j_1}\sigma_1}(r\mathrm{i})g^{\sigma_{j_2}\sigma_2}(r\mathrm{i})
\cdots g^{\sigma_{j_n}\sigma_n}(r\mathrm{i})\right|\\
&\geq&\left|g(r\mathrm{i})\right|^n\left(1-\sum_{\begin{smallmatrix}j_1j_2\cdots j_n\in S_n~\textrm{such that}\\
\sigma_{j_1}\sigma_{j_2}\cdots\sigma_{j_n}\neq \sigma_1^{-1}\sigma_2^{-1}\cdots\sigma_n^{-1}
\end{smallmatrix}}
\left|\frac{g^{\sigma_{j_1}\sigma_1}(r\mathrm{i})}{g(r\mathrm{i})}\right|\left|\frac{g^{\sigma_{j_2}\sigma_2}(r\mathrm{i})}{g(r\mathrm{i})}\right|
\cdots\left|\frac{g^{\sigma_{j_n}\sigma_n}(r\mathrm{i})}{g(r\mathrm{i})}\right|
\right)\\
&\geq&\left|g(r\mathrm{i})\right|^n\left(1-
\sum_{\begin{smallmatrix}j_1j_2\cdots j_n\in S_n~\textrm{such that}\\
\sigma_{j_1}\sigma_{j_2}\cdots\sigma_{j_n}\neq \sigma_1^{-1}\sigma_2^{-1}\cdots\sigma_n^{-1}
\end{smallmatrix}}\frac{1}{d!-1}
\right)\\
&&\textrm{by the fact $\sigma_{j_i}\sigma_i\neq\mathrm{Id}$ for some $1\leq i\leq n$ and (\ref{lmm})}\\
\\
&>&\left|g(r\mathrm{i})\right|^n\left(1-
\frac{n!-1}{d!-1}\right)\\
&\geq&0,
\end{eqnarray*}
which claims $\det(B)\neq0$. 
Therefore,
$g(\tau)$ is completely free in $\mathbb{C}(X(N))/\mathbb{C}(X(1))$, as desired. 
\end{proof}

\bibliographystyle{amsplain}

\begin{thebibliography}{10}

\bibitem {B-J} D. Blessenohl and K. Johnsen, \textit{Eine Versch\"{a}rfung des Satzes von der Normalbasis}, J. Algebra 103 (1986), no. 1, 141-–159.


\bibitem {Hachenberger} D. Hachenberger, \textit{Universal normal bases for the abelian closure of the field of rational numbers}, Acta Arith. 93 (2000), no. 4, 329–-341.

\bibitem {J-K-S}  H. Y. Jung, J. K. Koo and D. H. Shin, \textit{Normal bases of ray class fields over imaginary quadratic fields}, Math. Z. 271 (2012), no. 1--2, 109–-116.

\bibitem {K-L} D. Kubert and S. Lang, \textit{Modular Units}, Grundlehren der mathematischen Wissenschaften 244, Spinger-Verlag, New York-Berlin, 1981.

\bibitem {K-S} J. K. Koo and D. H. Shin, \textit{Completely normal elements in some finite abelian extensions}, Cent. Eur. J. Math. 11 (2013), no. 10, 1725–-1731.

\bibitem {Lang} S. Lang, \textit{Elliptic Functions}, 2nd edn, Grad. Texts in Math.
112, Spinger-Verlag, New York, 1987.

\bibitem {Leopoldt} H.-W. Leopoldt, \textit{Uber die Hauptordnung der ganzen Elemente eines abelschen Zahlk\"{o}rpers}, J. Reine Angew. Math. 201 (1959), 119-–149.

\bibitem {Okada} T. Okada, \textit{On an extension of a theorem of
S. Chowla}, Acta Arith. 38 (1980/81), no. 4, 341--345.

\bibitem {Schertz} R. Schertz, \textit{Galoismodulstruktur und elliptische Funktionen},
J. Number Theory 39 (1991), no. 3, 285–-326.

\bibitem {Shimura} G. Shimura, \textit{Introduction to the Arithmetic Theory of
Automorphic Functions}, Iwanami Shoten and Princeton University
Press, Princeton, NJ, 1971.

\bibitem {Taylor} M. J. Taylor, \textit{Relative Galois module structure of rings of integers and elliptic functions II}, Ann. of Math. (2) 121 (1985), no. 3, 519-–535.

\bibitem {vanderWaerden} B. L. van der Waerden, \textit{Algebra I}, Springer, New York, 1991


\end{thebibliography}

\address{
Department of Mathematical Sciences \\
KAIST \\
Daejeon 34141\\
Republic of Korea} {jkkoo@math.kaist.ac.kr}
\address{
Department of Mathematics\\
Hankuk University of Foreign Studies\\
Yongin-si, Gyeonggi-do 17035\\
Republic of Korea} {dhshin@hufs.ac.kr}
\address{
Department of Mathematical Sciences \\
KAIST \\
Daejeon 34141\\
Republic of Korea} {math\_dsyoon@kaist.ac.kr}

\end{document}